\documentclass[11pt]{article}
\usepackage{amsmath,amssymb,url,color}
\newcommand{\EQ}{\begin{equation}}
\newcommand{\EN}{\end{equation}}
\newtheorem{theorem}{Theorem}
\newtheorem{definition}{Definition}
\newtheorem{corollary}{Corollary}
\newtheorem{proposition}{Proposition}
\newtheorem{lemma}{Lemma}

\newtheorem{example}{Example}
\newtheorem{remark}{Remark}
\newtheorem{open}{Open Problem}

\newcommand{\FF}{{\mathbb F}}
\newcommand{\Fk}{{\mathbb F}_{2^k}}  
  
\newcommand{\F}{{\mathbb F}_{2^n}}

\newcommand{\Sp}{{\cal S}}

\newenvironment{proof}{\begin{trivlist}\item[]{\em Proof. }}%
{\samepage\hfill$\diamond$\end{trivlist}}

\title{Frobenius linear translators giving rise to \\ new infinite classes of permutations and bent functions}
\author{N. Cepak, E. Pasalic, and A. Muratovi\'{c}-Ribi\'{c}}
\begin{document}
\maketitle
\begin{abstract}
We show the existence of many infinite classes of permutations over finite fields and bent functions by extending the notion of linear translators, introduced by Kyureghyan \cite{kyu}. We call these translators Frobenius translators since the derivatives of $f:\FF_{p^n} \rightarrow \FF_{p^k}$, where $n=rk$, are of the form $f(x+u\gamma)-f(x)=u^{p^i}b$, for a fixed $b \in \FF_{p^k}$ and all $u \in \FF_{p^k}$, rather than considering the standard case corresponding to $i=0$. This considerably extends a rather rare family  $\{f\}$ admitting linear translators of the above form.
Furthermore, we solve a few open problems in the recent article \cite{CeChPa2016} concerning the existence and an exact specification of $f$ admitting classical linear translators, and an open problem introduced in \cite{Secondary} of finding a triple of bent functions $f_1,f_2,f_3$ such that their sum $f_4$ is bent and that the sum of their duals $f_1^*+f_2^*+f_3^*+f_4^*=1$.
Finally, we also specify two huge families of permutations over $\FF_{p^n}$ related to the condition that $G(y)=-L(y)+(y+\delta)^s-(y+\delta)^{p^ks}$ permutes the set $\Sp=\{\beta \in \FF_{p^n}: Tr_k^n(\beta)=0\}$, where $n=2k$ and $p>2$. Finally, we offer generalizations of constructions of bent functions from \cite{Sihem} and described some new bent families using the permutations found in \cite{CeChPa2016}.
\end{abstract}

\section{Introduction}
The main goal of this paper is to further extend the possibilities of employing the concept of linear translators, introduced by Kyureghyan \cite{kyu}, for the purpose of constructing new classes of permutation polynomials over finite fields  explicitly. Some of these permutations are then further used in constructing bent functions.
A finite field of order $p^n$ is denoted $\FF_{p^n}$ where $p$ is any prime and $n$ a positive integer.
A polynomial $F\in\FF_{p^n}[x]$ is said to be a permutation if its associated
mapping $x\mapsto F(x)$ over $\FF_{p^n}$ is bijective. During the last few years there has been a tremendous progress in construction methods and characterisation of many infinite classes of permutations, see a  survey on recent works in \cite{Hou14} and the references therein.

This paper extends the work in \cite{CeChPa2016} and \cite{Sihem}. In \cite{CeChPa2016}there are proposed several new classes of permutation polynomials 
\EQ\label{eq:form}
F~:~x\mapsto L(x)+L(\gamma) h(f(x)), ~f:\FF_{p^{rk}}\rightarrow \FF_{p^{k}},~
h:\FF_{p^{k}}\rightarrow \FF_{p^{k}},
\EN
which were originally studied by  Kyureghyan in \cite{kyu}. 

In \cite{Sihem} permutation polynomials of the form
\EQ\label{eq:form2}
F~:~x\mapsto L(x)+L(\gamma) (h(f(x)) + \frac{f(x)}{b}), ~f:\FF_{p^{rk}}\rightarrow \FF_{p^{k}},~
h:\FF_{p^{k}}\rightarrow \FF_{p^{k}},
\EN
are studied and then further used in the construction of bent functions.

Here $\gamma\in\FF^*_{p^{rk}}$ is a so-called 
{\it $b$-linear translator}
of $f$ (cf. Definition~\ref{de:tr}) and $L$ a linear permutation. It should be noted  that this construction
is in a certain sense a generalization of the so-called 
{\it switching construction} \cite{ChaKyu-ffa,ChKySu}. Akbary, Ghioca and
 Wang unified the Kyureghyan's construction  for arbitrary subsets $S \subset \FF_{p^{n}}$ (not only subfields of $\FF_{p^{n}}$) 
 along with proposing a few other constructions in \cite{AkGhWa11}. This general criterion is now 
called AGW criterion \cite[Theorem 8.1.39]{MullenWang}. 

After these pioneering works a series of papers  
\cite{TuZeLiHe15,TuZeHu14,TuZeJi15,YuDi} (among others) treated the same topic of specifying new classes of permutation polynomials of the form (\ref{eq:form}). For a nice survey of recent achievements related to this particular class of permutations  the reader is referred to \cite{Hou14}. In particular, the existence of linear translators were analyzed in \cite{CeChPa2016} for some simple polynomial forms (monomials and binomials) and their efficient embedding in (\ref{eq:form}) then resulted  in several explicit  classes of permutation polynomials. Apart form the unified framework provided by AGW criterion,  most of the recent  attempts were towards  specifying suitable functions $h,f$ and $L$ as in (\ref{eq:form}). Alternatively, for $F$ given by 
\begin{equation*}
F~:~x\mapsto L(x)+\gamma(f(x)+\delta)^s,~\delta\in\FF^*_{p^n},
\end{equation*}
 the main idea was to specify suitable degrees $s$, elements $\delta \in \FF_{p^n}$, and the function $f$ for some particular field characteristic $p$, see e.g. \cite{TuZeLiHe15}, thus only giving rise to sporadic families of permutations.

 The main obstacle when considering the forms (\ref{eq:form}) and (\ref{eq:form2}) is that some new classes of permutation polynomials could be specified provided the existence of suitable polynomials admitting linear translators. For instance, it was shown in \cite{CeChPa2016} that for $n=rk$ (where $r>1$), the function  $f(x)=\beta x^i+x^j$, $i<j$, where $f:\FF_{p^n} \rightarrow \FF_{p^k}$ and $\beta \in \FF^*_{p^n}$,  has a linear translator if and only if $n$ is even, $k=\frac{n}{2}$, and furthermore $f(x)=T^n_k(x)$. This indicates that the class of polynomials $f:\FF_{p^n} \rightarrow \FF_{p^k}$ admitting linear translators is quite likely rather small. To increase its cardinality and consequently to be able to derive other classes of permutation polynomials, we extend the original definition of linear translators to cover a wider class of functions admitting such translators. We  call these translators {\em Frobenius translators} since the derivative of $f$ is rather expressed as $f(x+u\gamma)-f(x)=u^{p^i}b$ in contrast to standard definition $f(x+u\gamma)-f(x)=ub$. Apparently, linear translators are just a special case of Frobenius translators. To justify this extension we may for instance consider the mapping  $f: x\mapsto T^n_k(x^{2^{\ell k}+1})$ over $\FF_{2^n}$, where $n=rk$ and  $1\leq \ell\leq r-1$, which does not have linear but admits a Frobenius translator, cf. Example~\ref{ex:1}. This gives us the possibility to construct permutation polynomials whose form greatly resembles  (\ref{eq:form}), (\ref{eq:form2}) though using Frobenius translators instead, cf. Theorem~\ref{th:extGohar}, Proposition~\ref{pro:Sihem}. In connection to the results in \cite{CeChPa2016}, we also address some existence issues for the classes of functions given by $f(x)= T^n_k(\beta x^{p^i+p^j})$, where $n=rk$, admitting linear translators and specify exactly the value of $\gamma$ in this case.   
 In addition, another  class of permutations of the form $F(x)=L(x)+(x^{p^k}-x+\delta)^s$ is proposed by specifying those $L$, $s$, and $\delta$ that satisfy the condition given recently in \cite{CeChPa2016}.

In the second  part of this article we focus on the use of suitable quadruples of bent functions  and Frobenius translators in order to provide new secondary  constructions of bent functions. Recently, many works have been devoted to secondary constructions of bent functions and for an exhaustive list of main contributions the reader is referred to  \cite{goodReference}. Here, we mainly focus on  the construction of Mesnager {\em et al.}  \cite{Mesnager1, Mesnager2}, where  bent functions are constructed using a suitable set of permutations whose duals  are also explicitly defined. This is a nice  property since in general computing the dual of a bent functions is a hard problem. Many secondary constructions rely on the initial bent functions whose duals satisfy certain properties. In \cite{Mesnager1}, the required property for three bent functions $f_1, f_2, f_3$, whose sum $f_4=f_1+f_2+f_3$ is  again bent, is that  $f_1^* + f_2^* + f_3^* + f_4^*=0$, where $f_i^* $ denotes the dual of $f_i^*$. This problem has been partially solved in \cite{Mesnager1} and a general method for finding quadruples of  so-called anti-self dual bent functions is given in \cite{chinesePaper}. Nevertheless, a slightly different approach \cite{Secondary} that uses  a quadruple of bent functions $f_1, \ldots, f_4$, sharing the above properties   but this time  $f_1^* + f_2^* + f_3^* + f_4^*=1$ instead, also leads to the design of secondary bent functions. The  problem of finding such quadruples satisfying $f_1^* + f_2^* + f_3^* + f_4^*=1$ was left open in \cite{Secondary}. We provide an efficient and generic solution to this problem which allows us to explicitly  specify further secondary classes of  bent functions and their duals. Based on the use of linear translators, in \cite{Mesnager2} the authors derived several infinite families of bent functions by defining suitable permutations from which initial quadruples of bent functions are defined. These results are generalized in a straightforward manner using Frobenius translators, thus offering a much wider class of secondary bent functions.   
 
 The rest of this article is organized as follows. The concept of linear translators and the generic method of specifying  new permutations based on their use is given in Section~\ref{sec:prel}. In Section~\ref{sec:Frob} we generalize the concept of linear translators by introducing the notion of Frobenius translators, which is proved useful for specifying some classes of permutations for those cases when linear translators are inefficient.
In Section~\ref{sec:gen}, we employ Frobenius translators to specify some secondary classes of bent functions and their duals.
Some concluding remarks are given in Section~\ref{sec:conc}.

\section{Linear translators - preliminaries}\label{sec:prel}
For clarity, we recall the original definition of linear translators given in \cite{kyu} by Kyureghyan. Throughout this article $p$ designates any prime and $n=rk$. 
\begin{definition}\label{de:tr}
Let $f$ be a function from $\FF_{p^n}$ to 
$\FF_{p^k}$, $\gamma\in\FF_{p^n}^*$
and $b$ fixed in $\FF_{p^k}$.
Then $\gamma$ is a $b$-{\it linear translator} for $f$ if
\[
f(x+u\gamma)-f(x)=ub~~\mbox{for all $x\in\FF_{p^n}$ and  all $u\in\FF_{p^k}$}.
\]
In particular, when  $k=1$, $\gamma$ is usually  said to be a 
 $b$-{\it linear structure} of the  function $f$ (where $b\in\FF_p$),
that is
\[
f(x+\gamma)-f(x)=b~~\mbox{for all $x\in\FF_{p^n}$}.
\]
\end{definition}
We denote by $Tr(\cdot)$ the absolute trace on $\F$ and by $T^n_k(\cdot)$
the trace function from $\FF_{p^n}$ to  $\FF_{p^k}$:
\[
T^n_k(\beta)=\beta+\beta^{p^k}+\dots+\beta^{p^{(n/k-1)k}}.
\]
We have also to recall  that a $\FF_{p^k}$-linear function 
on $\FF_{p^n}$ is of the type 
\[
L:\FF_{p^n} \rightarrow \FF_{p^n},~L(x)=\sum_{i=0}^{r-1}\lambda_ix^{p^{ki}}~,~
\lambda_i\in\FF_{p^n}.
\]
The following general theorem is given in \cite{kyu}
without proof since the proof is an equivalent of that given in
 \cite{ChaKyu-ffa} and \cite{ChSa11}, respectively, when $k=1$ and $k=n$.
\begin{theorem}
A   function $f$ from $\FF_{p^n}$ to 
$\FF_{p^k}$, has a linear translator if and only if there is a
 non-bijective $\FF_{p^k}$-linear function $L$ on $\FF_{p^n}$ such that
\[
f(x)=T^n_k\left(H\circ L(x)+\beta x\right)
\]
for some $H:\FF_{p^n}\rightarrow \FF_{p^n}$ and $\beta\in\FF_{p^n}$. In this case
the kernel of $L$ is contained in the subspace of linear translators
(including $0$ by convention). 
\end{theorem}

The construction of permutations based on linear translators, introduced by  Kyureghyan
in \cite[Theorem 1]{kyu}, is given below.
\begin{theorem}{\rm\cite[Theorem 1]{kyu}}\label{th:main}
Let $n=rk$, $k>1$. Let $L$ be a $\FF_{p^k}$-linear permutation on $\FF_{p^n}$.
Let $f$ be a function from $\FF_{p^n}$ onto $\FF_{p^k}$, 
$h:\FF_{p^k}\rightarrow\FF_{p^k}$, $\gamma\in\FF_{p^n}^*$,
and $b$ is fixed in $\FF_{p^k}$.

Assume that $\gamma$ is a $b$-linear translator of $f$. Then
$$F(x)=L(x)+L(\gamma)h(f(x))$$ permutes $\FF_{p^n}$ if and only if 
$g:u\mapsto u+bh(u)$ permutes $\FF_{p^k}$.
\end{theorem}
\section{Frobenius translators}\label{sec:Frob}
The main restriction of Theorem~\ref{th:main} is that it only gives new permutation polynomials for 
linear translators of $f$ satisfying the conditions in Definition~\ref{de:tr}.
\begin{example} \label{ex:1}
Let $p=2$, $n=rk$ and $f: x\mapsto T^n_k(x^{2^{\ell k}+1})$ with 
$1\leq \ell\leq r-1$. Let $\gamma\in \F$
and $u$ be any element of $\Fk$. Then
\begin{eqnarray*}
f(x)+f(x+u\gamma) &=& T^n_k\left(x^{2^{\ell k}+1}+(x+\gamma u)^{2^{\ell k}+1}\right)\\
&=&T^n_k\left(x^{2^{\ell k}}\gamma u+x(\gamma u)^{2^{\ell k}}+(\gamma u)^{2^{\ell k}+1}
\right)\\ &=& u~T^n_k\left(x(\gamma^{2^{\ell k}}+\gamma^{2^{n-\ell k}})\right)+
u^2T^n_k\left(\gamma^{2^{\ell k}+1}\right).
\end{eqnarray*}
This shows that $f(x)+f(x+u\gamma)=u^2~T^n_k(\gamma^{2^{\ell k}+1})$, for all $x$
and all $u \in \Fk$, if and only if $\gamma^{2^{\ell k}}+\gamma^{2^{n-\ell k}}=0$,
which is equivalent to $\gamma^{2^{2\ell k}}=\gamma$.
\end{example}
In the above example $b=T^n_k(\gamma^{2^{\ell k}+1})$ is not a linear translator of $f$ since we would obtain 
$f(x+\gamma u)+f(x)=u^2 b$, for $\gamma$ satisfying $\gamma^{2^{2\ell k}}=\gamma$, instead of having $ub$ on the right-hand side.
To find  other (not affine) functions $f$ which have $b$-translators appears to be
a difficult problem. The global description is
given in \cite[Section 2]{kyu} but to have precise instances would be useful
for some constructions. In particular, extending Definition~\ref{de:tr} to cover other cases, as illustrated in the above example, would be useful for deducing other families of permutation polynomials. 

To accomplish this we extend the definition of linear translators to cover the case when $f(x+\gamma u)-f(x)=u^{p^i} b$, as given below. 
\begin{definition}\label{de:tr2}
Let $n=rk$, $1\leq k\leq n$. Let $f$ be a function from $\FF_{p^n}$ to 
$\FF_{p^k}$, $\gamma\in\FF_{p^n}^*$
and $b$ fixed in $\FF_{p^k}$.
Then $\gamma$ is an $(i,b)$-\textnormal{Frobenius translator} for $f$ if
\[
f(x+u\gamma)-f(x)=u^{p^i}b~~\mbox{for all $x\in\FF_{p^n}$ and for all $u\in\FF_{p^k}$},
\]
where $i=0,\ldots,k-1$.
\end{definition} 

Notice that in the above definition taking $i=0$ gives a standard definition of translators. The next proposition generalizes the standard properties of linear translators  to the
case of Frobenius translators.

\begin{proposition}\label{pro:FTprop}
	Let $\gamma _1, \gamma _2 \in \FF_{p^n}$ be $(i,b_i)$ and $(i,b_2)$-Frobenius translators, respectively, of the function $f:\FF _{p^n} \rightarrow \FF_{p^k}$. Then
	\begin{itemize}
		\item $\gamma _1 + \gamma _2$ is an $(i, b_1+b_1)$-Frobenius translator of $f$,
		\item $c\gamma _1$ is a $(i,c^{p^i}b_1)$-Frobenius translator of $f$, for any $c\in \FF_{p^k}^*$.
	\end{itemize}
\end{proposition}
\begin{proof}
	\begin{eqnarray*}
		f(x+u(\gamma _1 + \gamma _2)) - f(x) &=& f(x + u\gamma _1) + u^{p^i}b_2 - f(x)\\
		&=& f(x) + u^{p^i}b_1 + u^{p^i}b_2 - f(x)\\
		&=& u^{p^i}(b_1+b_2)
	\end{eqnarray*}
	\begin{eqnarray*}
		f(x+u(c\gamma _1)) - f(x) &=& f(x + (uc)\gamma _1) - f(x)\\
		&=& (uc)^{p^i}b_1\\
		&=& u^{p^i}(c^{p^i}b_1)
	\end{eqnarray*}
\end{proof}

The Corollary below will be useful when satisfying conditions of constructions in Section \ref{sec:gen}.

\begin{corollary}\label{cor:3trans}
	In the binary case the sum of any three $(i,b)$-Frobenius translators $\gamma_1,\gamma_2,\gamma_3$, such that
	$\gamma_1+\gamma_2+\gamma_3\neq 0$, is again an $(i,b)$-Frobenius translator.
\end{corollary}
\begin{proof}
	By applying Proposition \ref{pro:FTprop} we know that $\gamma_1+\gamma_2+\gamma_3$ is a $(i,b+b+b)$-Frobenius translator. Since we are considering the binary case, that is an $(i,b)$-Frobenius translator.
\end{proof}

\begin{theorem}\label{th:extGohar}
For $n=rk$, let $h : \FF_{p^k} \rightarrow \FF_{p^k}$ be an arbitrary mapping and let $\gamma \in \FF_{p^n}$ be an $(i,b)$-Frobenius translator of $f : \FF_{p^n} \rightarrow \FF_{p^k}$
, that is $f(x+u\gamma)-f(x)=u^{p^i}b$ for all $x\in\FF_{p^n}$ and  all $u\in\FF_{p^k}$.
Then, the mapping 
\begin{equation} \label{eq:permFrob}
G (x) = L(x)^{p^i} + L(\gamma)^{p^i} h( f (x)),
\end{equation}
where $L:\FF_{p^n} \rightarrow \FF_{p^n}$ is an $F_{p^k}$-linear permutation,  permutes $\FF_{p^n}$ if and only if the mapping $g(u) = u + bh(u)$ permutes $\FF_{p^k}$.
\end{theorem}
\begin{proof} We follow the same steps as in the proof of \cite[Theorem 6]{kyu}. Let us first consider the special case $L(x)=x$, thus the function $F(x) = x^{p^i} + \gamma^{p^i} h( f (x))$. Assume that $x,y \in \FF_{p^n}$ satisfy $F(x)=F(y)$. Then
$$F (x) = x^{p^i} + \gamma^{p^i} h(f (x)) = y^{p^i} + \gamma^{p^i} h(f (y)) = F (y),$$
and hence
$$x^{p^i} = y^{p^i} + \gamma^{p^i} (h (f (y)) - h (f (x))) = y^{p^i} + \gamma^{p^i} a,$$
where $a=h (f (y)) - h (f (x)) \in \FF_q$. This is equivalent to saying that $x=y + \gamma a^{p^{n-i}}$, thus we suppose that $F(y)=F(y + \gamma a^{p^{n-i}})$. Then, using 
\begin{eqnarray*} F(y + \gamma a^{p^{n-i}})&=& y^{p^i} + {(\gamma a^{p^{n-i}})}^{p^i} + \gamma^{p^i} h(f(y + \gamma a^{p^{n-i}}))\\&=&y^{p^i} + \gamma^{p^i} a + \gamma^{p^i} h(f(y) +ab),
\end{eqnarray*}
we get 
\begin{equation*}
y^{p^i} + \gamma^{p^i} h(f (y)) = y^{p^i} + \gamma^{p^i} a + \gamma^{p^i} h(f(y) +ab),
\end{equation*}
which can be rewritten as
\begin{equation}\label{eq:heq}
 h(f (y)) =  a +  h(f(y) +ab).
\end{equation}
The mapping $F$ is a permutation of $\FF_{p^n}$ if and only if the only $a$ satisfying (\ref{eq:heq}) is $a = 0$. Using exactly the same arguments as in \cite{kyu}, one can conclude that $F$ is a permutation if and only if $g(u) = u + bh(u)$ permutes $\FF_{p^k}$. 

To show that $G(x)$ is a permutation it is enough to notice that $G(x)=L(F(x))$.
\end{proof}
%
\begin{remark}
The condition imposed on $h$, which applies to both linear and Frobenius translators, requiring that for a given $b$ the function 
$x + bh(x)$ is a permutation of $\FF_{p^k}$ is easily satisfied. Indeed, given any permutation $g$ over $\FF_{p^k}$ we can define 
$h(x) =1/b(g(x)-x)$ so that $x + bh(x)=g(x)$ is a permutation. Thus, the main challenge is to specify $\{f:\FF_{p^n}\rightarrow \FF_{p^k}\}$ which admit linear/Frobenius translators.  Each such translator then gives different permutations over $\FF_{p^n}$ for different permutations $g$ over $\FF_{p^k}$.
\end{remark}
Apart from Example \ref{ex:1}, one can for instance find Frobenius translators  by combining trace functions, more precisely by defining   $f(x)=Tr_k^n(x)+Tr_{2k}^n(x)$, for $n=4k$, as shown below.

\begin{proposition} For $n=4k$, the function $f:\FF_{p^n} \rightarrow \FF_{p^{2k}}$, defined by $f(x)=Tr_k^n(x)+Tr_{2k}^n(x)$,  always has a $0$-translator if $\gamma+ \gamma ^{p^{2k}}=0$. In the binary case, it also has a $(k,\gamma ^{p^k}+\gamma ^{p^{3k}})$-Frobenius translator.
\end{proposition}
\begin{proof}
	Let $n=4k$ and $f(x)=Tr_k^n(x)+Tr_{2k}^n(x)$. Let also $\gamma \in \FF _{p^{4k}} ^*$
	and $u\in \FF _{p^{2k}}$. Then
	\begin{eqnarray*}
	f(x+ u\gamma ) - f(x) &=& Tr_k^{4k}(x+u\gamma)+Tr_{2k}^{4k} (x+u\gamma )-Tr_k^{4k}(x)-Tr_{2k}^{4k}(x) \\
	&=& Tr_k^{4k}(x+u\gamma)+Tr_k^{4k}(-x) +Tr_{2k}^{4k} (x+u\gamma )+Tr_{2k}^{4k}(-x)\\
	&=& Tr_k^{4k}(u\gamma)+Tr_{2k}^{4k} (u\gamma )\\
	&=& 2u\gamma + (u\gamma)^{p^{k}} + 2(u\gamma)^{p^{2k}} + (u\gamma)^{p^{3k}}\\
	&=& 2u(\gamma + \gamma ^{p^{2k}}) + u^{p^k}(\gamma ^{p^k} + \gamma ^{p^{3k}}).
	\end{eqnarray*}
	
	For $p\neq 2$ the only possibility that $f$ has a linear translator is $\gamma + \gamma ^{p^{2k}}=0$,
	which results in a $0$-translator.
	In the binary case, we have $f(x+ u\gamma ) - f(x) = u^{2^k}(\gamma ^{2^k} + \gamma ^{2^{3k}})$, for any
	$x\in \FF _{2^{4k}}$ and any $u\in \FF _{2^{2k}}$, which means that $\gamma$ is a
	$(k,\gamma ^{p^k}+\gamma ^{p^{3k}})$-Frobenius translator.
\end{proof}

\subsection{Some existence issues}
In this section we specify exactly Frobenius translators for certain classes of mappings  $f: \FF_{p^n} \rightarrow \FF_{p^k}$ which gives us the possibility to specify some new infinite classes of permutations.  
The following existence results are similar to the ones presented in \cite{CeChPa2016}, with the difference that here
we consider Frobenius translators by means of Definition \ref{de:tr2}.

\begin{proposition}
	Let $f(x)=x^d$, $f: \FF_{p^n} \rightarrow \FF_{p^k}$, where $n=rk$ and $r>1$. Then the function $f$ does not have Frobenius translators in the sense of Definition~\ref{de:tr2}.
\end{proposition}
\begin{proof}
	This result follows directly from the proof of Proposition $1$ in \cite{CeChPa2016} by direct calculation.
\end{proof}
On the other hand, binomial mappings of the form $f(x)=\beta x^i+x^j$ still admit Frobenius translators as shown below.
\begin{proposition}\label{prop:binom}
	Let $f(x)=\beta x^i+x^j$, $i<j$, where $f:\FF_{p^n} \rightarrow \FF_{p^k}$, $\beta \in \FF^*_{p^n}$ and $n=rk$,
	where $r >1$. Then the function $f$ has a linear translator $\gamma$ if and only if $n$ is even, and $k=\frac{n}{2}$.
	Furthermore, $f(x)=x^{p^{i'}} + x^{p^{i' + \frac{n}{2}}}$ and
	$\gamma$ is an $(i',\gamma^{p^{i'}} + \gamma^{p^{i'+\frac{n}{2}}})$-linear translator. 
\end{proposition}
\begin{proof}
	The same method as in \cite[Proposition $2$]{CeChPa2016}, that uses  Lucas' Theorem and the formula
	relating the coefficients of a given function and the coefficients of its derivative \cite{EnesDeriv}, is used to prove that for $f$
	to have a translator (either linear or Frobenius) we necessarily have $i=p^{i'} $ and $j=p^{j'}$, for some $i'$ and $j'$.
	
	Let us now analyse $f(x)=\beta x^{p^{i'}} + x^{p^{j'}}$.
	Since $f$ maps to a subfield $\FF_{p^k}$,
	the following must  be satisfied for all $x$:
	\begin{eqnarray*}
	(\beta x^{p^{i'}} + x^{p^{j}})^{p^k} - \beta x^{p^{i'}} - x^{p^{j'}}=0\\
	\beta^{p^{k}} x^{p^{i'+k}} + x^{p^{j'+k}}- \beta x^{p^{i'}} - x^{p^{j'}}=0.
	\end{eqnarray*}
	Hence, the exponents $\{p^{i'+k},p^{j'+k},p^{i'},p^{j'}\}$ cannot be two by two
	distinct. This forces $p^{i'+k} \equiv p^{j'}\mod (p^n-1)$ and $p^{j'+k} \equiv p^{i'}\mod (p^n-1)$.
	It follows that $j'=i'+k, k=\frac{n}{2}$ and $\beta=1$.
	 Then,
	\begin{eqnarray*}
	f(x+u\gamma) - f(x) &=& (x+u\gamma)^{p^{i'}} + (x+u\gamma)^{p^{i'+\frac{n}{2}}}\\
	&=& u^{p^{i'}}\gamma^{p^{i'}} + u^{p^{i'+\frac{n}{2}}}\gamma^{p^{i'+\frac{n}{2}}}\\
	&=&u^{p^{i'}}(\gamma^{p^{i'}} + \gamma^{p^{i'+\frac{n}{2}}})
	\end{eqnarray*}	
	and $\gamma$ is an $(i',\gamma^{p^{i'}} + \gamma^{p^{i'+\frac{n}{2}}})$-linear translator.
\end{proof}

We conclude this section by specifying exactly Frobenius translators related to quadratic mappings of the form $f(x)= T^n_k(\beta x^{p^i+p^j})$ as discussed  in \cite{CeChPa2016}.

\begin{lemma}\cite{CeChPa2016} \label{le:lt}
	Let $n=rk$ and $f(x)= T^n_k(\beta x^{p^i+p^j})$, where $i < j$.
	Then, $f$ has a derivative independent of $x$, that is,
 $f(x+u\gamma)-f(x)=T^n_k (\beta (u\gamma)^{p^i+p^j} )$
	for all $x\in \FF_{p^n}$,
	all $u\in \FF _{p^k}$, if and only if $\beta,\gamma \in \FF^*_{p^n}$
 are related through, 
	\begin{equation}\label{eq:quadtrace}
	\beta \gamma^{p^{i+lk}}+\beta^{p^{(r-l)k}}\gamma^{p^{i+(r-l)k}}=0,
	\end{equation}
	where $0 < l < r$ satisfies $j=i+kl$.
\end{lemma}

Nevertheless, the relation between $\beta$ and $\gamma$ imposed by (\ref{eq:quadtrace}) and their existence were not investigated in \cite{CeChPa2016}.  Below, we specify the exact relationship between $\beta$ and $\gamma$, thus implying the possibility of defining some infinite classes of permutations explicitly. 

\begin{proposition}\label{le:abmod}
	Let $n,r,k,l$ be as in Lemma \ref{le:lt}, $\alpha$ be a primitive element of $\FF _{p^n}$, and
	$\gamma= \alpha ^a, \beta = \alpha ^b \in \FF_{p^n}$. Then
	$$\beta \gamma^{p^{i+lk}}+\beta^{p^{(r-l)k}}\gamma^{p^{i+(r-l)k}}=0$$
	if and only if
	\[ b= \left \{ \begin{array}{ll}
	 -ap^{i+lk}(p^{(r-l)k}+1)\mod (p^n-1), & p=2 \\
	 -ap^{i+lk}(p^{(r-l)k}+1) + \frac{p^n-1}{2}(1-p^{(r-l)k})^{-1} \mod (p^n-1), & p \neq 2.
	 \end{array} \right.
	 \]
\end{proposition}
\begin{proof}
	Expressed in terms of $\alpha$, the  equation $$-\alpha^{b+ap^{i+lk}}=\alpha^{bp^{(r-l)k} + ap^{i+(r-l)k}}$$
	is considered separately for the binary and  non-binary case. Let $p=2$. In this case $$ \alpha^{b+ap^{i+lk}}=\alpha^{bp^{(r-l)k} + ap^{i+(r-l)k}}.$$
	Therefore,
	\begin{eqnarray*}
	b+ap^{i+lk} \mod (p^n-1)&=&bp^{(r-l)k} + ap^{i+(r-l)k} \mod (p^n-1)\\
	b(1-p^{(r-l)k}) \mod (p^n-1) &=& ap^{i+lk}(p^{2(r-l)k}-1)\mod (p^n-1) \\
	b(1-p^{(r-l)k}) \mod (p^n-1) &=& ap^{i+lk}(p^{(r-l)k}-1)(p^{(r-l)k}+1)\mod (p^n-1) \\
	b \mod (p^n-1) &=& -ap^{i+lk}(p^{(r-l)k}+1)\mod (p^n-1) \\
	b &=& -ap^{i+lk}(p^{(r-l)k}+1)\mod (p^n-1).
	\end{eqnarray*}
	
	Let $p\neq 2$.
	In this case $-1=\alpha^{\frac{p^n-1}{2}}$ and
	$$ \alpha^{\frac{p^n-1}{2}}\alpha^{b+ap^{i+lk}}=\alpha^{bp^{(r-l)k} + ap^{i+(r-l)k}}.$$
	Therefore,
	\begin{eqnarray*}
	b+ap^{i+lk} + \frac{p^n-1}{2}\mod (p^n-1)&=&bp^{(r-l)k} + ap^{i+(r-l)k} \mod (p^n-1)\\
	2b(1-p^{(r-l)k}) \mod (p^n-1) &=& 2a(p^{i+(r-l)k} - p^{i+lk} ) \mod (p^n-1) \\
	2b(1-p^{(r-l)k}) \mod (p^n-1) &=& 2ap^{i+lk}(p^{2(r-l)k}-1) \mod (p^n-1) \\
	2b(1-p^{(r-l)k}) \mod (p^n-1) &=& 2ap^{i+lk}(p^{(r-l)k}-1)(p^{(r-l)k}+1) \mod (p^n-1)\\
	2b &=& -2ap^{i+lk}(p^{(r-l)k}+1)\mod (p^n-1).
	\end{eqnarray*}
\end{proof}

The Frobenius translators related to the  function $f$  in Lemma \ref{le:lt} are further specified in the result below.

\begin{theorem}\label{th:lt}
	Let $n=rk$ and  $f(x)=T^n_k(\beta x^{p^i+p^{i + kl}})$, where $r>1$ and $0 <l <r$.
	Assume that $\gamma\in \FF^*_{p^n}$ is an $(s,b)$-translator of $f$, where $b=T^n_k (\beta \gamma^{p^i+p^{i+lk}})$.
	Then:
	\begin{enumerate}
	\item[i)] If $p=2$ the condition (\ref{eq:quadtrace}) in Lemma \ref{le:lt} must be satisfied and $s=i+1$.
	In particular, if $\beta \in \FF_{2^k}$ then $\gamma=1$ is a
	$0$-translator of $f$ if $r$ is even, and $\gamma=1$ is an $(i+1,\beta)$-translator
	if $r$ is odd.
	\item[ii)] If $p >2$ we necessarily have $b=0$. In  particular,
	if $\beta \in \FF_{p^k}$ then $n$ is even and $\gamma$ must satisfy
	$\gamma^{p^{2kl}-1}=-1$ and $Tr_k^n(\gamma^{p^i+p^{i+lk}})=0$.
	\end{enumerate}
\end{theorem}

\begin{proof}
	If (\ref{eq:quadtrace}) is satisfied then
	
	\begin{eqnarray*}
	f(x+u\gamma)-f(x) &=& u^{2p^i}T^n_k \left(\beta \gamma^{p^i+p^{i+lk}}\right).
	\end{eqnarray*}
	
	$i)$ Let $p=2$. Then $u^{2p^i}=u^{p^{i+1}}$ and $\gamma$ is an $(i+1,b)$-translator. 
	In particular, if $\beta \in \FF_{2^k}$ then $\gamma=1$ is a solution
	to (\ref{eq:quadtrace}).  Then,  
	$b=\beta T^n_k(\gamma^{2^i+2^{i+lk}})=\beta T^n_k(1)=0$ if $r$ is even and
	$b=\beta$ for odd $r$.

	$ii)$ For $p >2$ we have $2p^{i}\equiv p^t \pmod{p^k-1}$ for some positive integer $t$, which implies
	$2p^i = m(p^k-1)+p^j$. Since $p$ is odd, the left-hand side of the equation
	is even and the right-hand side is odd, which is impossible. The only remaining option is for $\gamma$
	to be a $0$-translator. 
	
 	The rest follows directly from \cite[Theorem $4$]{CeChPa2016}.
\end{proof}

The following example specifies  a function having a linear translator constructed in this way.
\begin{example}
	Let us consider  $f(x)= T^n_k(\beta x^{p^i+p^j})$ given in Lemma \ref{le:lt}, where $p=2$. The relevant parameters are: $n=rk=8,r=4,k=2$
	and $i=2, l=1, j=i+kl=4$.
	Let $\alpha$ be a primitive element of the field $\FF _{2^4}$. We fix an arbitrary element $\gamma = \alpha ^a$ by setting e.g.  $a=3$.
	Now the function  $f: \FF _{2^8} \rightarrow \FF _{2^2}$, having a linear translator, can be specified 
	using the condition (\ref{eq:quadtrace}) in Lemma \ref{le:lt}.
	The element $\beta = \alpha ^b$ is then computed, using Proposition \ref{le:abmod},  by specifying  $b$  to be
	$$b=-ap^{i+lk}(p^{(r-l)k}+1)\mod (p^n-1)=-3\cdot2^{2+1\cdot2}\cdot(2^{(4-1)\cdot2}+1) \mod (255)=195.$$
	By Theorem \ref{th:lt}, it follows that $f(x)=T ^n_k (\beta x^{p^i+p^j})=T ^8_2(\alpha ^{195} x^{2^2+2^4})$ has an
	$(s,b)=(3, T^8_2 (\alpha ^{195}\alpha ^{3(2^2+2^4)}) )$-Frobenius translator. 
\end{example}

\section{Permuting subspaces and derived permutations} \label{sec:subspace}
In this section we consider a special class of polynomials for which the permutation property is scaled down  to the same property though restricted to a certain subspace of the field $\FF_{p^n}$. It will be shown that the special form considered here and the restriction of the permutation property to this subspace leads us easily to a large class of permutations of the form $F(x)=L(x)+(x^{p^k}-x+\delta)^s$. 

We recall the following result that was derived recently in  \cite{CeChPa2016}.
\begin{theorem}\label{th:var2}
Let $p$ be an odd prime,
$n=2k$ and $F:\FF_{p^n}\rightarrow \FF_{p^n}$ with 
\begin{equation}\label{eq:suspperm}
F(x)=L(x)+(x^{p^k}-x+\delta)^s, ~\delta\in\FF_{p^n},
\end{equation}
where $L\in \FF_{p^k}[x]$ is a linear permutation 
and $s$ is any integer in the range $[0,p^n-2]$. 
Then $F$ is a permutation over $\FF_{p^n}$ if and only if
the function $G$ 
\[
G(y)=-L(y)+(y+\delta)^s-(y+\delta)^{p^ks},
\]
 is a permutation of the subspace $\Sp=\{y\in\FF_{p^n}~|~T^n_k(y)=0\}$.
 In particular, if $s$ satisfies $p^k s  \equiv s \pmod{p^n-1}$ 
 then $F$ is a permutation.
\end{theorem}
We notice that the form of $F$ above corresponds to $x + bh(x)$ when $L(x)=x$ and $b=1$. 
Furthermore, as already noticed in \cite{CeChPa2016}, $L$ induces a permutation of $\Sp$. By noting that $Tr^n _k(\alpha)=0$ if and only if there exists $\beta \in \FF _{p^n}$ such that $\alpha=\beta - \beta ^{p^k}$, we can write $\Sp=\{y\in\FF_{p^n}~|~T^n_k(y)=0\} = \lbrace \beta - \beta^{p^k} \vert \beta \in \FF _{p^n}\rbrace$. 
Clearly, $G:\Sp \rightarrow \Sp$ since $\Sp$ is a subspace and $(y+\delta)^s-(y+\delta)^{p^ks} \in \Sp$. 



We  first consider  the special case when $\delta \in \Sp$.

\begin{proposition} \label{pro:perSp}
	Let $p$ be odd, $n=2k$, and $\Sp=\{y\in\FF_{p^n}~|~T^n_k(y)=0\}$. Then the mapping
	$$G(x)=-L(x)+(x+\delta)^s - (x +\delta)^{p^ks} $$ permutes the set $\Sp$ for any $\delta \in \Sp$, any linear
	permutation 	$L$, and any even $s\in \lbrace 2,4,\ldots , p^n-1 \rbrace$. Consequently, 
	$$F(x)=L(x)+(x^{p^k}-x+\delta)^s,$$
	is a permutation for any $\delta \in \Sp$, for any $L$ and any even $s\in \lbrace 2,4,\ldots , p^n-1 \rbrace$.
\end{proposition}
\begin{proof}
	Since $s$ is even, let us write $s=2s'$ and let $a\in \Sp$ be  arbitrary.
	Then because $a\in \Sp$ we can write $a=b-b^{p^k}$ for some $b\in \FF _{p^n}$ and
	\begin{eqnarray*}
		(b-b^{p^k})^{2s'p^k}&=&(b^{p^k}-b^{p^{2k}})^{2s'}\\
		&=& (b^{p^k}-b)^{2s'}\\
		&=&(-(b^{p^k}-b))^{2s'}\\
		&=&(b^{p^k}-b)^{2s'}.
	\end{eqnarray*}
	Since $x+\delta$ is an element of $\Sp$ for every $x,\delta \in \Sp$, the function $G(x)$, restricted to $\Sp$, can be
	also written as
	\begin{eqnarray*}
		G(x)&=&-L(x)+(x+\delta)^{2s'}-(x+\delta)^{2s'p^k}\\
		&=&-L(x)+(x+\delta)^{2s'}-(x+\delta)^{2s'}\\
		&=&-L(x).
	\end{eqnarray*}
	Since $L(x)$ is a linear permutation and we already observed that it induces permutation on $\Sp$,
	$G(x)$ must be a permutation of $\Sp$. From Theorem \ref{th:var2}, it then follows that
	$F(x)=L(x)+(x^{p^k}-x+\delta)^s$ is a permutation.

\end{proof}
This results provides us with many infinite classes of permutations of the form (\ref{eq:suspperm}), as illustrated by the following example. 

\begin{example}
	Let $p=3, n=2k, k=3, L(x)$ be any $\FF_{3^3}$-linear permutation polynomial of $\FF_{3^6}$, and
	$\delta \in \FF _{3^6}$ be such that $Tr_3^6(\delta)=0$.
	It then follows from Proposition \ref{pro:perSp} that the mapping
	$$G(x)=-L(x) + (x+\delta )^s - (x +\delta)^{p^ks}$$ permutes the set $\Sp = \lbrace y \in \FF _{3^6} \vert Tr_3^6(y)=0
	\rbrace$ for any even $s$. Further, by Theorem \ref{th:var2} 
	$$ F(x) = L(x)+(x^{3^3}-x+\delta)^s$$
	is  a permutation for any $\delta \in \Sp$ and any even $s$.
\end{example}

A closely related issue in this context is whether there are suitable $L(y)$ and exponents $s$ when $\delta \not \in \Sp$ . 

\begin{proposition}\label{pro:perSpdelta} 
	Let $p$ be odd, $n=2k$, and $\Sp=\{y\in\FF_{p^n}~|~T^n_k(y)=0\}$. Then the mapping
	$$G(x)=-L(x)+(x+\delta)^s - (x +\delta)^{p^ks} $$ permutes the set $\Sp$ for any $\delta$, any
	linearized permutation $L$, and any $s=t(p^k+1)$, where $t$ is an integer.
	Consequently, 
	$$F(x)=L(x)+(x^{p^k}-x+\delta)^{t(p^k+1)},$$
	is a permutation for any $\delta$, for any $L$, and any integer $t$.
\end{proposition}
\begin{proof}
	For every $x\in \FF _{p^n}$ we can see that
	\begin{eqnarray*}
		x^{t(p^k+1)} - x^{p^kt(p^k+1)}&=& x^{t(p^k+1)} - x^{tp^{2k} +tp^k}\\
		&=& x^{t(p^k+1)} - x^t x^{tp^k}\\
		&=&x^{t(p^k+1)} - x^{t(p^k+1)}\\
		&=&0.
	\end{eqnarray*}
	It follows that 
	$$G(x)=-L(x)+(x+\delta)^{t(p^k+1)} - (x+\delta)^{p^kt(p^k+1)}=-L(x). $$
	
	Similarly as before, it follows from Theorem \ref{th:var2} that $G(x)$ is a permutation of $\FF _{p^n}$.
\end{proof}

\section{Application to bent functions}\label{sec:gen}

In this section we  provide a generalization of results in  \cite{Sihem}  by using  Frobenius translators instead of standard linear translators, when $p=2$. This allows to specify some new infinite classes of permutations and their inverses similarly to the approach in \cite{Sihem} which in turn gives rise to suitable quadruples of permutations from which  secondary classes of bent functions can be deduced. Furthermore, we also solve  an open problem \cite{Secondary} mentioned in the introduction which concerns the existence of quadruples of bent functions whose duals sum to one. 

\subsection{Generalization of certain permutations using Frobenius translators} 
The main result of the method in \cite{Mesnager1}  is the condition imposed on the duals of four bent functions $f_1, \ldots, f_4$ (where $f_4=f_1+f_2+f_3$) given by $f_1^* + f_2^* + f_3^*+ f_4^*=0$, where $f_i^*$ denotes the dual of $f_i$. This condition was shown to be both necessary and sufficient in order that the function $H=f_1f_2+f_1f_3 + f_2f_3$ is bent. This naturally leads to the employment of the Maiorana-McFarland class of bent functions, where a bent function $f_j:\FF_{2^{n}} \times \FF_{2^{n}} \rightarrow \FF_2$ in this class is defined as $f_j(x,y)= Tr_1^n(x\phi_j(y) + \theta_j(y))$, for some permutation $\phi_j$ over $\FF_{2^n}$ and arbitrary function $\theta_j$ over $\FF_{2^n}$. It was shown in \cite{Mesnager2} that the above quadruples of bent functions are easily identified using a set of permutations defined by means of linear translators. We show that this approach is easily extended to cover Frobenius translators as well, which induces larger classes of these sets of permutations suitable to define new bent functions. 

\begin{proposition}[Generalization of Proposition $3$, \cite{Sihem}]\label{pro:Sihem}
	Let $f:\FF _{2^n} \rightarrow \FF_{2^k}$, let $L:\FF_{2^n} \rightarrow \FF_{2^n}$ be an
	$\FF_{2^k}$-linear permutation of $\FF_{2^n}$, and let $g:\FF_{2^k}\rightarrow \FF_{2^k}$ be a permutation.
	Assume $\gamma \in \FF_{2^n}^*$ and $a\in \FF_{2^k}^*$
	are such that $\gamma$ is an $(a,i)$-Frobenius translator of $f$ with respect to $\FF_{2^k}$. Then  the function
	$\phi : \FF _{2^n}\rightarrow \FF_{2^n},$
	
	\begin{equation} \label{eq:perm_gen}
		\phi = L(x) + L(\gamma)\left( g(f(x)) + \frac{f(x)}{a} \right)^{2^{n-i}}, 
	\end{equation}
	is a permutation polynomial of $\FF_{2^n}$ and
	
	$$\phi ^{-1} = L^{-1}(x) + \gamma a^{2^i} \left( g^{-1}\left( \frac{f(L^{-1}(x))}{a}\right) + f(L^{-1}(x)) \right) ^{2^{n-i}}.$$
\end{proposition}
\begin{proof}
	Let us define $h:\FF_{2^n} \rightarrow \FF_{2^n}$ as
	$$ h(x)=x+\gamma\left( g(f(x)) + \frac{f(x)}{a}\right) ^{2^{n-i}}.$$
	
	Then, setting  $y=x+\gamma \left( g(f(x)) + \frac{f(x)}{a}\right)^{2^{n-i}}$ leads to
	$$f(y)=f\left( x+\gamma \left( g(f(x)) + \frac{f(x)}{a}\right)^{2^{n-i}} \right) = f(x) + a \left( g(f(x)) + \frac{f(x)}{a}\right)^{2^{n-i}2^i} =ag(f(x)).$$
	
	Therefore, $f(x)=g^{-1}\left( \frac{f(y)}{a}\right)$ and	
	$$x=y+\gamma \left( g(f(x)) + \frac{f(x)}{a}\right)^{-2^i} = y + \gamma a^{-2^{n-i}}\left (f(y) + g^{-1}\left( \frac{f(y)}{a}\right) \right)^{2^{n-i}}.$$
	
	This means that $h$ is a permutation of $\FF_{2^n}$ and its inverse is
	$$h^{-1}(x)= x + \gamma a^{-2^{n-i}}\left (f(x) + g^{-1}\left( \frac{f(x)}{a}\right) \right)^{2^{n-i}}.$$
	
	Now we can define $\phi$ as $\phi = L \circ h,$
	$$\phi(x) = L(h(x))=L\left( x + \gamma \left( g(f(x)) + \frac{f(x)}{a}\right)^{2^{n-i}}\right) =
	L(x) + L(\gamma)\left( g(f(x)) + \frac{f(x)}{a}\right)^{2^{n-i}}, $$
	and $\phi ^{-1}$ as
	$$\phi^{-1}(x) = h^{-1}\circ L^{-1} = L^{-1}(x) + \gamma a^{-2^{n-i}}\left (f(L^{-1}(x)) + g^{-1}\left( \frac{f(L^{-1}(x))}{a}\right) \right)^{2^{n-i}}. $$
\end{proof}

In order to use these permutations in constructing new secondary classes of bent functions they must satisfy the condition $(\mathcal{A}_n)$, which was first introduced by Mesnager in \cite{Mesnager1} and later employed in \cite{Mesnager2}.

\begin{definition}
	Three pairwise distinct permutations $\phi_1, \phi _2, \phi _3$ of $\FF_{2^n}$ are said to satisfy $(\mathcal{A}_n)$ if the following conditions hold:
	\begin{itemize}
		\item $\psi = \phi_1 + \phi _2 + \phi _3$ is a permutation of $ \FF_{2^n}$,
		\item $\psi ^{-1} = \phi _1^{-1} + \phi_2^{-1} + \phi_3^{-1}.$
	\end{itemize}
\end{definition}

The main challenge is to define suitable permutations $\phi_i$ as in (\ref{eq:perm_gen}) so that $\psi =\phi _1 +  \phi _2 + \phi _3$ is also a permutation satisfying the condition $(\mathcal{A}_n)$, quite similarly to the approach taken  in \cite{Sihem}. To achieve this, the simplest way is to use the same  $L,f,g$ for all $\phi_j, j\in \lbrace 1,2,3\rbrace$, where the functions $\phi_i$ only differ in the term $L(\gamma_i)$. More precisely, the function $f$ admits different  $(a,i)$-Frobenius translators $\gamma _i$, for some fixed $i$ and $a$,  with the additional  condition that $\gamma_1+\gamma_2+\gamma_3$ is also an $(a,i)$-Frobenius translator of $f$.

In the non-binary cases, finding such triples of Frobenius translators can be difficult, but in the binary case, the sum of any three $(a,i)$-Frobenius translators is again an $(a,i)$-Frobenius translator, as Corollary \ref{cor:3trans} proves.


Then
$$\psi (x) = L(x) + L(\gamma_1 + \gamma_2 + \gamma_3)\left( g(f(x)) + \frac{f(x)}{a}\right)^{2^{n-i}},$$
$$\psi ^{-1}(x)= L^{-1}(x) + (\gamma_1 + \gamma _2 + \gamma _3 )a^{-2^{n-i}}\left (f(L^{-1}(x)) + g^{-1}\left( \frac{f(L^{-1}(x))}{a}\right) \right)^{2^{n-i}},$$
and it is easily verified that the permutations $\phi_j$ satisfy the condition $(\mathcal{A}_n)$. This approach allows us to  construct new bent functions using the  result from \cite{Mesnager1, Mesnager2} below.

\begin{proposition}[\cite{Mesnager1, Mesnager2}]\label{pro:gen1}
	Let $\phi_1, \phi_2, \phi_3$ be three pairwise distinct permutations satisfying $(\mathcal{A}_n)$. Then, the Boolean function $H: \FF_{2^n} \times \FF_{2^n}\rightarrow \FF_2$ defined by
	$$H(x,y) = Tr_1^n(x\phi_1(y))Tr_1^n(x\phi_2(y)) + Tr_1^n(x\phi_1(y))Tr_1^n(x\phi_3(y))  + Tr_1^n(x\phi_2(y))Tr_1^n(x\phi_3(y))$$
	is bent. Furthermore, its dual function $H^*$ is given by
	$$ H^*(x,y)=Tr_1^n(\phi_1^{-1}(x)y)Tr_1^n(\phi_2^{-1}(x)y) + Tr_1^n(\phi_1^{-1}(x)y)Tr_1^n(\phi_3^{-1}(x)y)  + Tr_1^n(\phi_2^{-1}(x)y)Tr_1^n(\phi_3^{-1}(x)y).$$
\end{proposition}
Notice that $H$ is essentially defined as $H=f_1f_2 + f_1f_3 + f_2f_3$, where $f_j(x,y)=Tr_1^n(x\phi_j(y))$ so that $\theta_j(y)=0$.  

\begin{remark}
	Using the same techniques the following Propositions and Theorems from \cite{Sihem} can be generalized as well with minor modifications.
	\begin{itemize}
		\item Theorems $1,2,3,4$ in \cite{Sihem};
		\item Propositions $4,5,6$ in \cite{Sihem}.
	\end{itemize}
\end{remark}

Due to similarity, we only discuss a generalization of  Theorem 1 in \cite{Sihem} and give an example of  bent functions constructed using this generalization.

\begin{theorem}[Generalized Theorem $1$, \cite{Sihem}]\label{the:con1}
	Let $f:\FF_{2^n}\rightarrow \FF _{2^k}$, let $L:\FF _{2^n}\rightarrow \FF_{2^n}$ be an $\FF_{2^k}$-linear permutation
	of $\FF_{2^n}$, and let $g:\FF_{2^k}\rightarrow \FF_{2^k}$ be a permutation. Assume $\gamma _1, \gamma _2,
	\gamma _3\in \FF _{2^n}^*$ are all pairwise distinct $(a,i)$-Frobenius translators of $f$ with respect to
	$\FF_{2^k}$ ($a\in \FF_{2^k}^*$) such that $\gamma_1+\gamma_2+\gamma_3$ is again an $(a,i)$-Frobenius translator. Suppose $\gamma_1 + \gamma_2 + \gamma_3 \neq 0$. Set $\rho (x)=
	\left( g(f(x))+\frac{f(x)}{a} \right) ^{2^{n-i}}$ and $\tilde{\rho} (x)=a^{2^i}\left( g^{-1}\left( \frac{f(x)}{a}\right)
	+ f(x) \right) ^{2^{n-i}}$. Then,
	\begin{eqnarray*}
		H(x,y)&=& Tr(xL(y))+ Tr(L(\gamma _1)x\rho (y)) Tr(L(\gamma _2)x\rho(y)) + \\
		& & Tr(L(\gamma _1)x\rho (y)) Tr(L(\gamma _3)x\rho(y)) + 
		Tr(L(\gamma _2)x\rho (y)) Tr(L(\gamma _3)x\rho(y))
	\end{eqnarray*}
	is bent. Furthermore, its dual function $H^*$ is given by
	\begin{eqnarray*}
		H^*(x,y)&=& Tr(yL^{-1}(x)) +
		Tr(\gamma _1 y \tilde{\rho}(L^{-1}(x)))Tr(\gamma _2 y \tilde{\rho}(L^{-1}(x))) + \\
		& & Tr(\gamma _1 y \tilde{\rho}(L^{-1}(x)))Tr(\gamma _3 y \tilde{\rho}(L^{-1}(x))) + 
		Tr(\gamma _2 y \tilde{\rho}(L^{-1}(x)))Tr(\gamma _3 y \tilde{\rho}(L^{-1}(x))).
	\end{eqnarray*}
\end{theorem}
\begin{proof}
	The only difference between Theorem $1$ \cite{Sihem}, and the generalized version presented here is the
	modification to $\rho $ and $\tilde{\rho}$. In the original approach $\rho (x)=
	\left( g(f(x))+\frac{f(x)}{a} \right)$ and $\tilde{\rho} (x)=a^{2^i}\left( g^{-1}\left( \frac{f(x)}{a}\right)
	+ f(x) \right)$. Then, raising  $\rho $ and $\tilde{\rho}$ to the power of $2^{n-i}$, as it has been done in	the proof of Proposition \ref{pro:Sihem},  the proof of Theorem \ref{the:con1} is the same as the proof of Theorem $1$, \cite{Sihem}.
\end{proof}

\begin{example}\label{ex:3Frobenius}
	Let $n=8$,  $\omega$ be a primitive element of $\FF_{2^8}$,  $L$ be an arbitrary $\FF_{2^4}$-linear
	permutation of $\FF_{2^8}$ and $h$ be an arbitrary permutation of 
	$\FF_{2^4}$. Suppose we want the function $f: \FF_{2^8}\rightarrow \FF_{2^4}$ to be a binomial and to use it in
	the construction of a bent function using Theorem \ref{the:con1}. Using only the standard definition of a linear
	translator, we would be forced to define $f(x)=Tr^8_4(x)$ according to Proposition $2$ from \cite{CeChPa2016}.  But using
	Proposition \ref{prop:binom} we can define $f(x)=x^{2^i} + x^{2^{i+4}}$ for any $i$ with any $\gamma\in \FF_{2^8}$
	being an $(\gamma^{2^i} + \gamma^{2^{i+4}} ,i)$-Frobenius translator of $f$.
	
	To use Theorem \ref{the:con1}, we need to define three pairwise distinct $(a,i)$-Frobenius translators. So we need to
	find three distinct $\gamma_1, \gamma _2, \gamma_3$ such that
	$$ \gamma_1^{2^i} + \gamma_1^{2^{i+4}} = \gamma_2^{2^i} + \gamma_2^{2^{i+4}} = \gamma_3^{2^i} + \gamma_3^{2^{i+4}} = (\gamma_1+\gamma_2+\gamma_3)^{2^i} + (\gamma_1+\gamma_2+\gamma_3)^{2^{i+4}}=a.$$
	This would imply that $\gamma_1, \gamma _2, \gamma_3, \gamma_1+\gamma_2+\gamma_3$ are all $(a,i)$-Frobenius translators. A quick computation shows that $\gamma_1 + \gamma _2, \gamma_1 + \gamma_3, \gamma_2+\gamma_3 \in \FF_{2^4}$ is required.
	We select $\gamma_1 = \omega, \gamma_2 = \omega^3, \gamma_3=\omega^{16}$ and, for example, if we fix $i=2$, we get
	$$\gamma_1^{2^i} + \gamma_1^{2^{i+4}} = \gamma_2^{2^i} + \gamma_2^{2^{i+4}} = \gamma_3^{2^i} + \gamma_3^{2^{i+4}} = (\gamma_1+\gamma_2+\gamma_3)^{2^i} + (\gamma_1+\gamma_2+\gamma_3)^{2^{i+4}} =\omega^{136}$$
	and $\omega+\omega^3+\omega^{16}=\omega^{48}\neq 0$.
	
	Let $\rho, \tilde{\rho}$ and $H$ be defined as in Theorem \ref{the:con1}. It follows that $H$ is a bent function. 
\end{example}

\subsection{New bent functions from suitable quadruples of bent functions}\label{sec:openp}
In difference to the above approach, which preserves the variable space of input functions, another method of constructing secondary bent functions on the extended variable space was recently proposed in \cite{Secondary}. Nevertheless, quite a similar set of conditions on initial bent functions $f_1,f_2,f_3$, which was left as an open problem in \cite{Secondary},  is imposed in order that the resulting function $F$ defined on a larger variable space is bent. 
\begin{open} \cite{Secondary} \label{open}
	Find such bent functions $f_1,f_2,f_3$ that $f_1+f_2+f_3 = f_4$ is again a bent function and $f_1^* + f_2^* + f_3^* + f_4^*=1$.
\end{open} 

The design rationale is illustrated  by Example 4.9 \cite{Secondary}, where using $f_1,f_2,f_3: \FF_{2^n} \rightarrow \FF_2$ that satisfy the above condition, implies that $F:\FF_{2^n}\times\FF_2\times \FF_2 $ defined as
$$F(X,y_1,y_2)=f_1(X) + y_1(f_1+f_3)(X) + y_2(f_1+f_2)(X)$$
is bent. 

Below we present a construction that solves this open problem and gives an example of its use.

\begin{theorem}
	Let $f_i(X)=f_i(x,y) =Tr(x\phi_i (y)) + h_i(y) $ for $i \in \lbrace 1,2,3\rbrace$, where $\phi _i$ satisfy the condition $(\mathcal{A}_n)$ and $x, y \in \FF_{2^{n/2}}$.
If  the functions $h_i$ satisfy 
\begin{equation}\label{eq:cond_theta}
h_1(\phi_1^{-1}(x)) + h_2(\phi_2^{-1}(x)) + h_3(\phi_3^{-1})(x)) + (h_1 + h_2 + h_3)((\phi_1 + \phi_2 + \phi_3)^{-1}(x))= 1,
\end{equation}
 then $f_1,f_2,f_3$ are solutions to  Open Problem \ref{open}.
\end{theorem}
\begin{proof}
	Let $f_4=f_1+f_2+f_3 = Tr(x (\phi_1 + \phi_2 + \phi_3)(y)) + (h_1+h_2+h_3)(y)$. Since the permutations $\phi_i$ satisfy
	the condition $(\mathcal{A}_n)$, their sum is again a permutation and $f_4$ is a bent
	Maiorana-McFarland function. Its dual is
	$$f_4^*=Tr(y(\phi_1 + \phi_2 + \phi_3)^{-1}(x)) + (h_1+h_2+h_3)((\phi_1 + \phi_2 + \phi_3)^{-1}(x)).$$
	
	Then,
	\begin{eqnarray*}
	f_1^* + f_2^* + f_3^* + f_4^*&=& Tr(y(\phi_1^{-1}(x))) + h_1(\phi_1^{-1}(x)) +\\
	&& + Tr(y(\phi_2^{-1}(x))) + h_2(\phi_2^{-1}(x)) +Tr(y(\phi_3^{-1}(x))) + h_3(\phi_3^{-1}(x)) + \\
	&& + Tr(y (\phi_1 + \phi_2 + \phi_3)^{-1}(x)) + (h_1+h_2+h_3)(\phi_1 + \phi_2 + \phi_3)^{-1}(x)) \\
	&=&Tr(y((\phi_1^{-1} + \phi_2^{-1} + \phi_3^{-1} + (\phi_1 + \phi_2 + \phi _3)^{-1})(x))) + \\
	&& + h_1(\phi_1^{-1}(x)) + h_2(\phi_2^{-1}(x)) + h_3(\phi_3^{-1}(x)) + (h_1+h_2+h_3)((\phi_1 + \phi_2 + \phi_3)^{-1}(x))\\
	&=& h_1(\phi_1^{-1}(x)) + h_2(\phi_2^{-1}(x)) + h_3(\phi_3^{-1})(x)) + (h_1 + h_2 + h_3)((\phi_1 + \phi_2 + \phi_3)^{-1}(x))\\
	&=& 1.
	\end{eqnarray*}
\end{proof}
The following example illustrates the procedure of defining three suitable bent functions on $\FF_{2^n}$ used to specify a bent function $F$ on $\FF_{2^n}\times \FF_2\times\FF_2$. The condition (\ref{eq:cond_theta}) imposed on $h_i$ in the definition of suitable $f_i(x,y)=Tr(x \phi_i(y)) + h_i(y)$ turns out to be easily satisfied. 
\begin{example}
	
	Let $\alpha$ be a primitive element of $\FF_{2^6}$.
For simplicity, we define the permutations $\phi_i$ over $\FF_{2^6}$ as 
	$$\phi_1(y) = y + \alpha, \hspace{1cm} \phi_1(y) = y + \alpha^2, \hspace{1cm} \phi_1(y) = y + \alpha^3,$$
which are self-inverse and it is straightforward to verify that they satisfy the condition $(\mathcal{A}_n)$.
Define the Boolean functions $h_2,h_3:\FF_{2^6} \rightarrow \FF_2$ as

%

	$$ h_2(y)=0, \hspace{1cm} h_3(y)=1.$$
	After, we define the Boolean function $h_1$ in such a way that
	\begin{eqnarray*}
	h_1(\phi_1^{-1}(y)) + h_2(\phi_2^{-1}(y)) + h_3(\phi_3^{-1})(y)) + (h_1 + h_2 + h_3)((\phi_1 + \phi_2 + \phi_3)^{-1}(y))&=& 1\\
	h_1(\phi_1^{-1}(y)) + (h_1)((\phi_1 + \phi_2 + \phi_3)^{-1}(y))&=& 1\\
	h_1(y + \alpha) + (h_1)(y + \alpha + \alpha ^2 + \alpha ^3)&=&1.
	\end{eqnarray*}
	This condition is easily satisfied. We just construct the truth table of the Boolean function $h_1$ in
	such a way that for every $y\in \FF_{2^6}$ we have $h_1(y) = h_1(y+\alpha ^2+\alpha ^3)+1$.
	Now we construct bent Maiorana-McFarland functions $f_i:\FF_{2^6} \times \FF_{2^6} \rightarrow \FF_2,
	f_i(x,y)=Tr(x\phi_i(y)) + h_i(y)$
	and use them in the construction from Example $4.9$,  \cite{Secondary}.
	
	We define $F:\FF_{2^{12}}\times\FF_2 \times \FF_2 \rightarrow \FF_2,$
	$$F(X,y_1,y_2)=f_1(X) + y_1(f_1+f_3)(X) + y_2(f_1+f_2)(X) .$$
	The function $F$ was implemented and tested using the programming package Magma. It was confirmed that $F$ is a bent function.
\end{example}

\begin{remark}
In \cite[Remark 3]{chinesePaper}, a method to define anti-self-dual bent functions $f_1,f_2,f_3, f_1+f_2+f_3$ (thus $f_i^* = f_i+1$) is given which implies that $f_1^* + f_2^* + f_3^+ + f_4^*=0$. Another construction of $f_1,f_2,f_3$ that satisfies this condition can be found in \cite[Section 5]{PSCons},  where $f_1,f_2,f_3$ all belong to the partial spread ($\mathcal{PS}$) class of Dillon \cite{Dillon}. It is based on a well-known property of the $\mathcal{PS}$ class that the dual $f^*$  of a $\mathcal{PS}$ function $f$ is defined by substituting all the disjoint $\frac{n}{2}$-dimensional subspaces in its support by their orthogonal subspaces \cite{CarletBook}. It follows that $f_4^* = f_1^* + f_2^* + f_3^*$ and consequently $f_1^*+f_2^*+f_3^* + f_4^*=0$.
\end{remark}

\subsection{Some new infinite families of bent functions}\label{sec:fromParis}

In \cite{CeChPa2016} many infinite families of permutations based on linear translators were introduced,
some of which were already generalized in previous sections. It turns that in the binary case some of
those families satisfy the condition $(\mathcal{A}_n)$.

\begin{proposition}[\cite{CeChPa2016}]\label{pro:LTParis}
	Let $k>1 (n=rk)$, $f:\FF _{2^n} \rightarrow \FF_{2^k}$, $g:\FF _{2^k}\rightarrow \FF _{2^k}$, and let $\gamma$ be a 
	$0$-linear translator. Then
	$$ F(x)=x+\gamma g(f(x))$$
	is an involution.
\end{proposition}

Note that if $\gamma$ is a $0$-translator it is irrelevant to differentiate between linear and Frobenius translators.

\begin{proposition}
	Let $\gamma _1, \gamma_2, \gamma _3$ be pairwise distinct $0$-linear translators, and let $F_i(x)=x+\gamma_i g(f(x))$ for
	$i\in \lbrace 1,2,3 \rbrace$. Then the functions $F_i$ satisfy the condition $\mathcal{A}_n$.
\end{proposition}

\begin{proof}
	By Proposition \ref{pro:FTprop}, $\gamma_1+\gamma_2+ \gamma_3$ must again be a $0$-linear translator.
	\begin{eqnarray*}
	F_1(x) + F_2(x) + F_3(x) &=& x+\gamma _1 g(f(x)) + x+\gamma_2 g(f(x)) + x+\gamma_3 g(f(x))\\
	&=& x+(\gamma_1 + \gamma _2 + \gamma_3) g(f(x))
	\end{eqnarray*}
	Then, by Proposition \ref{pro:LTParis}, $F_1 + F_2 + F_3$ is again a permutation and an involution.
	This immediately implies that the second requirement of condition $(\mathcal{A}_n)$ is satisfied as well.
\end{proof}

It therefore follows that we can use the above presented permutations in constructing new families of bent functions, as was done in Proposition \ref{pro:gen1}. Since the proof also follows the same steps it is in this case skipped.

\begin{theorem}
	Let $k>1 (n=rk)$, $f:\FF _{2^n} \rightarrow \FF_{2^k}$, $g:\FF _{2^k}\rightarrow \FF _{2^k}$, and let
	$\gamma _i$ be pairwise distinct $0$-linear translators. Then
	
	\begin{eqnarray*}
	H(x,y) &=& Tr(xy) + Tr(\gamma_1 g(f(y)))Tr(\gamma_2 g(f(y))) +  Tr(\gamma_1 g(f(y)))Tr(\gamma_3 g(f(y))) + \\
	& & + Tr(\gamma_2 g(f(y)))Tr(\gamma_3 g(f(y)))
	\end{eqnarray*}
	
	is a self-dual bent function.
\end{theorem}
Another family of permutations that turns out to satisfy the condition $(\mathcal{A}_n)$ was introduced  in \cite{CeChPa2016}:
\begin{corollary}[\cite{CeChPa2016}]
	Let $k>1 (n=rk)$, $L$ be any $\FF_{2^k}$-linear permutation,  $f(x)=T_k^n(\beta x)$
	such that $Tr(\beta\gamma)=0$. Then the functions
	$$ F(x)=L(x)+L(\gamma) g(Tr_k^n(\beta x))$$
	are permutations for any $g:\FF _{2^k}\rightarrow \FF _{2^k}$. Moreover,
	$$F^{-1}(x) = L^{-1}(x) + L(\gamma)g(Tr_k^n(\beta L^{-1}(x))) .$$	
\end{corollary}

In a similar way as before we can show that $F_i(x) = L(x)+L(\gamma _i) g(Tr_k^n(\beta x))$ satisfy the condition $(\mathcal{A}_n)$ if $Tr_k^n(\gamma _i\beta)=0$. It follows that these permutations can also be used in constructing new families of bent functions.
\begin{theorem}
	Let $L$ be any $\FF_{2^k}$-linear permutation,$f(x)=T_k^n(\beta x)$, $g:\FF _{2^k}\rightarrow \FF _{2^k}$, and let
	$\gamma _i$ be such that $Tr_k^n(\gamma_i \beta)=0$. Then
	
	\begin{eqnarray*}
	H(x,y) &=& Tr(xL(y)) + Tr(L(\gamma_1)g(Tr_k^n(\beta x)))Tr(L(\gamma_2)g(Tr_k^n(\beta x)))+\\
	& & + Tr(L(\gamma_1)g(Tr_k^n(\beta x)))Tr(L(\gamma_3)g(Tr_k^n(\beta x))) + \\
	& & + Tr(L(\gamma_2)g(Tr_k^n(\beta x)))Tr(L(\gamma_3)g(Tr_k^n(\beta x)))
	\end{eqnarray*}
	
	is a bent function and its dual is
	
	\begin{eqnarray*}
	\tilde{H}(x,y) &=& Tr(yL^{-1}(x)) + Tr(L(\gamma_1)g(Tr_k^n(\beta L^{-1}(x))))Tr(L(\gamma_2)g(Tr_k^n(\beta L^{-1}(x))))+\\
	& & + Tr(L(\gamma_1)g(Tr_k^n(\beta L^{-1}(x))))Tr(L(\gamma_3)g(Tr_k^n(\beta L^{-1}(x)))) + \\
	& & + Tr(L(\gamma_2)g(Tr_k^n(\beta L^{-1}(x))))Tr(L(\gamma_3)g(Tr_k^n(\beta L^{-1}(x)))).
	\end{eqnarray*}
\end{theorem}

\section{Conclusions}\label{sec:conc}
In this article several classes of permutations and bent functions are derived using the concepts of linear and Frobenius translators. These Frobenius translators allow us to specify suitable sets of permutations based on which many new secondary classes of bent functions and their dual can be derived. The most interesting open problem in our opinion regards the existence of non-quadratic functions admitting linear/Frobenius translators. It might be the case that there are only a few classes of quadratic mappings having this kind of translators, discussed in \cite{CeChPa2016} and in this article, which are suitable for this type of construction.  

\section{Acknowledgements}
Enes Pasalic is partly supported by the Slovenian Research Agency (research program P3- 0384 and research project J1-6720).
Nastja Cepak is supported by the Slovenian Research Agency (research 25 program P3-0384 and Young Researchers Grant).\\

\end{document}